\documentclass[a4, 10pt,leqno]{amsart}
\usepackage{pifont}
\usepackage{mathtools}
\usepackage{mathabx}
\usepackage{bbm}
\usepackage{amsthm}
\usepackage{mathtools}
\usepackage{mathtools,accents}
\usepackage{amssymb,amscd,amsthm,enumerate}
\usepackage{mathrsfs}
\usepackage{a4wide}
\usepackage{eqnarray}
\usepackage[utf8]{inputenc}
\usepackage[T1]{fontenc}
\usepackage{lmodern}
\usepackage[english]{babel}
\usepackage{microtype}
\usepackage{float}
\usepackage{esint}
\usepackage{physics}
\usepackage{a4wide}
\usepackage{algorithmic}
\usepackage[ruled]{algorithm2e}
\usepackage[active]{srcltx}
\usepackage{lipsum}
\usepackage{upgreek}
\usepackage{listings}
\usepackage{wasysym}
\usepackage{enumerate}
\usepackage{color}
\usepackage[usenames,dvipsnames]{xcolor}
\usepackage{xcolor}
\definecolor{ultrablue}{rgb}{0.0,0.0, 1}
\definecolor{jigari}{rgb}{0.39,0.0, 0.0}
\usepackage{url}
\usepackage{hyperref}
\hypersetup{
	linktoc=page,
	colorlinks,
	linkcolor={jigari},
	citecolor={jigari},
	urlcolor={jigari}
}
\hypersetup{breaklinks=true}
\usepackage{soul}
\usepackage{graphicx}
\usepackage{tikz}
\usepackage{tikzpagenodes}
\usepackage{scrlayer-scrpage}
\usepackage{tikz}
\usetikzlibrary{calc,spy}
\usetikzlibrary{shapes.geometric, arrows}
\usepackage{pgfplots}
\usetikzlibrary{intersections, pgfplots.fillbetween}
\usepackage{tikz}
\usetikzlibrary{matrix}
\usepackage[all]{xy}
\usepackage{subfig}
\newlength{\myeqskip}  
\AtBeginDocument{%
	\setlength\abovedisplayskip{\myeqskip}%
	\setlength\belowdisplayskip{\myeqskip}%
	\setlength\abovedisplayshortskip{\myeqskip-\baselineskip}%
	\setlength\belowdisplayshortskip{\myeqskip}}
\allowdisplaybreaks
\numberwithin{equation}{section}
\usepackage[a4paper,bindingoffset=0in,%
left=1.2in,
right=1.2in,
top=1.4in,
bottom=1.4 in,%
footskip=1mm]{geometry}
\theoremstyle{plain}
\newtheorem*{theorem*}{Main Theorem}
\newtheorem{lemma}{Lemma}[section]
\newtheorem{theorem}[lemma]{Theorem}

\theoremstyle{definition}

\newtheorem{remark}[lemma]{Remark}

\theoremstyle{remark}

\newcounter{example}
\newenvironment{example}[1]{\refstepcounter{example}\par\medskip
	\noindent \textsc{\small Example~\theexample. #1} \rmfamily\hspace{-2pt}}{\medskip}
\makeatletter
\def\@xnamedef#1{\expandafter\protected@xdef\csname #1\endcsname}
\def\no@harm{} 
\def\ead@au#1{\protected@edef\@ead@au{#1}}
\patchcmd\runningauthor@fmt{\global\edef}{\protected@xdef}{}{}
\patchcmd\runningauthor@fmt{\global\edef}{\protected@xdef}{}{}
\patchcmd\author@fmt{\edef}{\protected@edef}{}{}
\patchcmd\add@xtok{\xdef}{\protected@xdef}{}{}
\makeatother
\newcommand{\secref}[1]{\textsection\thinspace\ref{#1}}

\newcommand{\R}{\mathbb R}

\DeclareMathOperator{\Hess}{Hess}

\DeclareMathOperator{\Riem}{Riem}

\newcommand{\K}{\mathcal K}
\newcommand{\y}{\footnotesize \textbf{\textit{y}}}

\newcommand{\dist}{\mathrm{d}}

\newcommand{\cpt}{{\textsf{CPT}}\,}
\usepackage{accents}
\newcommand*{\dt}[1]{%
	\accentset{\mbox{\bfseries .}}{#1}}
\newcommand*{\ddt}[1]{%
	\accentset{\mbox{\bfseries .\hspace{-0.25ex}.}}{#1}}
\newcommand*{\dddt}[1]{%
	\accentset{\mbox{\bfseries .\hspace{-0.25ex}.\hspace{-0.25ex}.}}{#1}}

\makeatletter
\newcommand*\bigcdot{\mathpalette\bigcdot@{.5}}
\newcommand*\bigcdot@[2]{\mathbin{\vcenter{\hbox{\scalebox{#2}{$\m@th#1\bullet$}}}}}
\makeatother
\newcount\stylenum  \newdimen\styledim 
\def\varstyle#1{\mathchoice{\stylenum=0 #1}{\stylenum=1 #1}{\stylenum=2 #1}{\stylenum=3 #1}}
\def\mathaxis{\fontdimen22\ifcase\stylenum 
	\textfont\or\textfont\or\scriptfont\or\scriptscriptfont\fi2 }
\def\setstyledim{\styledim=\ifcase\stylenum .1em\or.1em\or.07em\or.05em\fi\relax}
\def\sqdot{\mathbin{\varstyle{\raise\mathaxis\hbox{\setstyledim
				\kern\styledim 
				\vrule width1.2\styledim height.6\styledim depth.6\styledim
				\kern\styledim}}}}

		\DeclareMathAlphabet{\mathdutchcal}{U}{dutchcal}{m}{n}
		\SetMathAlphabet{\mathdutchcal}{bold}{U}{dutchcal}{b}{n}
		\DeclareMathAlphabet{\mathdutchbcal}{U}{dutchcal}{b}{n}
		\DeclareSymbolFont{myletters}{OML}{ztmcm}{m}{it}
		\DeclareMathSymbol{\nicelambda}{\mathord}{myletters}{"15}
		\usepackage{nicefrac}
\usepackage{stackengine}
\setstackEOL{\\}
\newcounter{tmpctr}
\newcommand\fancyRoman[1]{%
	\setcounter{tmpctr}{#1}%
	\setbox0=\hbox{\kern.2pt\textsf{\Roman{tmpctr}}}%
	\setstackgap{S}{-.6pt}%
	\Shortstack{\rule{\dimexpr\wd0+.1ex}{.7pt}\\\copy0\\
		\rule{\dimexpr\wd0+.1ex}{.7pt}}%
}

		\usepackage{pifont}
		
\newcommand{\restr}{\raisebox{-.1908ex}{$\big|$}}
\makeatother
\newcommand{\ident}{\raisebox{0pt}{\scalebox{1.1}{$\mathbbm{1}$}}\hspace{-1pt}}

\makeatletter
\newcommand{\tpitchfork}{%
	\vbox{
		\baselineskip\z@skip
		\lineskip-.52ex
		\lineskiplimit\maxdimen
		\m@th
		\ialign{##\crcr\hidewidth\smash{$-$}\hidewidth\crcr$\pitchfork$\crcr}
	}%
}
\makeatletter
\newcommand{\tpmod}[1]{{\@displayfalse\pmod{#1}}}
\makeatother

\usepackage{stackengine,scalerel}

\newcommand\overstar[1]{\ThisStyle{\ensurestackMath{%
			\setbox0=\hbox{$\SavedStyle#1$}%
			\stackengine{0pt}{\copy0}{\kern0\ht0\smash{\SavedStyle\star}}{O}{c}{F}{T}{S}}}}

\newcommand\dunderline[3][-1pt]{{%
		\sbox0{#3}%
		\ooalign{\copy0\cr\rule[\dimexpr#1-#2\relax]{\wd0}{#2}}}}
\makeatletter
\@namedef{subjclassname@2020}{%
	\textup{2020} Mathematics Subject Classification}
\makeatother
\makeatletter 
\def\l@subsection{\@tocline{2}{0pt}{3pc}{6pc}{}}
\makeatother
\makeatletter 
\def\l@subsection{\@tocline{2}{0pt}{3pc}{6pc}{}}
\makeatother

\def\bysame{\leavevmode\hbox to3em{\hrulefill}\thinspace}
\begin{document}
\title[\scriptsize {The rigidity of conformal circle-preserving transformations on Berwaldian manifolds}]{\small The rigidity of conformal circle-preserving transformations\\ on Berwaldian manifolds} 

%
\author[\protect \scriptsize Z. Fathi]{Zohreh Fathi}
\author[\protect \scriptsize S. Lakzian]{Sajjad Lakzian$^{^{\scalebox{1}{$\star$}}}$}
\address{\noindent -- Z. Fathi \& S. Lakzian \newline \noindent School of Mathematics, \newline Institute for Research in Fundamental Sciences (IPM), \newline P. O. Box 19395-
	5746, Tehran, Iran}
\address{\noindent -- S. Lakzian \newline \noindent Department of Mathematical Sciences\newline Isfahan University of  Technology (IUT) \newline Isfahan 8415683111, Iran.}
\email{\href{mailto:slakzian@iut.ac.ir}{slakzian@iut.ac.ir}}
\thanks{ZF is partially supported by IPM grant No. 1403530043 and SL is partially supported by IPM grant No. 1405530313.}
\thanks{SL is supported by the INSF Grant No. 4030556, awarded by the “On
	the Frontiers of Mathematical Sciences” program.}
\thanks{$\raisebox{2pt}{$\star$}$\textit{the corresponding author}}
\maketitle
%
\begin{abstract}
\par \textsl{We prove that a complete Berwaldian manifold $\left(M,F\right)$ admitting a nontrivial conformal circle preserving transformation (\cpt for short) must be Riemannian, provided that it has a dense subset on which no flag curvature vanishes (in particular, if $(M,F)$ has positive or negative flag curvature).}
\end{abstract}
\date{\today}
\section{Introduction}
\par A geodesic circle in a Riemannian manifold, $\left(M^n,g\right)$, is a curve with constant first geodesic curvature $\kappa \ge 0$, and vanishing second geodesic curvature~\cite{NM}; also, see~\secref{subsec:geod-circ}. 
\par A circle preserving (also called concircular) transformation (\cpt for short) is a diffeomorphism $\varphi$ between two Riemannian manifolds (or domains within) that preserves geodesic circles. The fundamentals of the theory of {\cpt}\!s in the Riemannian setting were established by Yano in~\cite{Yano1}--\cite{Yano5}, followed by important contributions from~\cite{Ish-Tash,Ish, Tash1, Tash2, VO, Kan, Kuh}.
\par By Vogel's theorem, a Riemannian \cpt must be conformal, that is, the pullback metric by $\varphi$ is $\widebar{g} = \rho^{-2} g$, where the conformal factor $\rho^{-1}$ satisfies the local PDE system, 
\begin{align*}
(\rho^{-1})_{i;j} - (\rho^{-1})_i(\rho^{-1})_j = \lambda g_{ij},
\end{align*}
for some scalar function $\lambda$~\cite{Br,VO}; here, semicolon ($;$) means covariant differentiation. When $\lambda$ is constant, $\varphi$ is a homothety and will be called a trivial \cpt. 
\par Writing in coordinate-free manner, the nonlinear geometric PDE describing a \cpt is
\begin{align*}
\nabla^2 \left(\rho\right) =\nicefrac{1}{n} \, \Delta \left(\rho\right) g, \quad \rho>0,
\end{align*}
or if we work with the logarithmic conformal factor (to absorb the constraint $\rho>0$ in the equation), $\sigma = - \ln \rho$, the PDE becomes
\begin{align}\label{eq:pde-sigma}
\nabla^2 \sigma - d\sigma\otimes d\sigma =\nicefrac{1}{n} \left(\Delta \sigma - \|\nabla \sigma\|^2\right) g.
\end{align}
\subsection*{Finslerian {\cpt}\!s}
\par In Finslerian structures, most geometric objects are defined on $TM$ rather than on $M$ itself. Nevertheless, geodesic circles can be defined in the same fashion as in the Riemannian case, with the caveat that, the Cartan connection (along the lift of the underlying curve) is the best one suited for this purpose (as there are several other ones). Indeed, metric compatibility ensures that the covariant derivative of a projectable vector field along a curve $c$ remains projectable. Consequently, successive derivatives give rise to a Frenet-Serret-Jordan frame; for further details, see~\cite[Chapter 4]{BF}. 
\par \dunderline{1.2pt}{\small \textit{\textbf{\textsf{Convention:}}}} \textsl{In what follows, we omit mentioning the diffeomorphism $\varphi$ and write a \cpt as $\bar{F} = e^{\sigma} F$. Mostly, we work with the PDE in terms of $\sigma$; however, when more convenient, we write quantities in terms of $\rho = e^{-\sigma}$.} 
\subsubsection{Conformality is not automatic!}\label{sec:conformality}
In the Finslerian setting, conformality of a \cpt fails to be automatic; the following counterexamples illustrate this. 
\begin{example}[Shen-Yang~\cite{SY}]
	In 2D, there exist infinitely many Minkowskian Randers metrics that admit \cpt\!s which are not conformal. To demonstrate their existence, for a constant $b\neq 0$, consider the Minkowskian Randers Finslerian length,
	\begin{align*}
		F (\y) = \sqrt{\|\y\|_{\sf Euc}} + (b,0)\;\;{\large \dt{}}_{_{\,\sf Euc}} \; \y.
	\end{align*}
	Let $V$ be the vector field in $\R^2$ that generates rotations around the origin. Thanks to the criterion established in~\cite[Theorem~5.1]{SY}, we know that $V$ is a circle-preserving vector field (its flow $\Phi_t$ consists of Finslerian \cpt\!s) but it is not a conformal vector field since $\mathcal{L}_{V^c}g$ fails to be a scalar multiple of $g$ where $V^c$ is the following lifted vector field to $TT\R^2$:
	\begin{align*}
		V^c = V^i\partial_i + \left<\grad^{\y} V^i, \y \right>_{\!\y}\dt{\partial}_i, \quad 1\le i\le 2.
	\end{align*}
	Consequently, for an infinite subset $D\subset \R$ (with $0$ as a limit point), the maps $\Phi_t$ with $t \in D$ are Finslerian \cpt\!s but fail to be conformal. For more details, see~\cite[Example 6.1]{SY}.
\end{example}

\begin{theorem*}
	Suppose $(M,F)$ is a complete Berwaldian manifold and the set of points where some of the flag curvatures vanish, is a nowhere dense set (in particular, if it is null in the Lebesgue sense in charts). If $(M,F)$ admits a nontrivial Finslerian conformal \cpt, $\bar{F} = e^{\sigma} F$, then it must be Riemannian. 
\end{theorem*}
\section{Preliminaries}
\subsection{Some fundamentals of Finslerian structures}\label{subsec:basic}
Throughout these notes, $M$ is a smooth manifold equipped with a Finslerian fundamental function $F(x,{\y})$ ($(x,{\y}) \in T_xM$, $\y = y^i \partial_i$). The associated Finslerian metric is denoted by $g_{ij}(x,\y)$ (alternatively we also use $g^{\y}$ or $\left<\cdot,\cdot  \right>_{\y}$ or simply $g$) and the Cartan tensor is denoted by $C_{ijk}(x,\y)$.
\subsubsection{Gradients of functions and normals of regular hypersurfaces}\label{subsubsec:grad}
\par The gradient of a function $f(x)$ w.r.t. $g^{\y}$ is given by 
$
	\grad^{\y} f(x) = f^r\partial_r
$ 
in other words,
$
df(X) = g^{\y}(\grad^{\y}f, X)
$,
holds for all $X$. Thus, $\nicefrac{\grad^{\y} f}{F(\grad^y f)}$ is the unique direction along which $f$ has the steepest ascent; consequently $\nicefrac{\grad^{\y} f}{F(\grad^y f)}$ is independent of $\y$. 
\par One can select a preferred direction $\y$ (to compute the gradient w.r.t.) by setting
$
\grad f := \ell^{-1}(df);
$
where $\ell$ is Legendre transformation; see~\cite[Chapter 3]{Shen-Lectures}. In other words, $\grad f$ is the unique vector satisfying
$
df = g^{\grad f}\left(\grad f, \,\cdot\, \right)
$,
and
$
\grad^{\grad f}f = \grad f
$. It is not difficult to see that $\grad f$ is perpendicular to the regular level sets of $f$ w.r.t. the metric $g^{\grad f}$; see~\cite{Shen-Lectures}. 
\subsection{The Cartan connection}\label{subsec:cart-con}
The geodesic spray coefficients, $G^i$, and the nonlinear connection coefficients, $G^i_j$, are defined as follows:
\begin{align*}
	G^s = \nicefrac{1}{4} \; g^{sl}\left( y^k\dt{\partial}_l\partial_k\left(F^2\right) - \partial_l\left(F^2\right)\right),\quad G^i_j := \dt{\partial}_j G^i.
\end{align*}
\par The double tangent vectors 
$
	\delta_i := \partial_i - G_i^r\dt{\partial}_r,
$
and
$
\dt{\partial}_i := \partial_{y^i}
$,
span the the horizontal sub-bundle, $\mathcal{H}_{(x,\y)}$ of $TTM$ and the vertical sub-bundle, $\mathcal{V}_{(x,\y)}$ respectively. The Christoffel symbols (symmetric) for the nonlinear part of the Cartan connection are
\begin{align*}
	\overstar{\Gamma}^k_{ij} := \nicefrac{1}{2}\; g^{kl}\left( \delta_i g_{jl} + \delta_j g_{il} - \delta_l g_{ij} \right),
\end{align*}
and the vertical connection coefficients are given by the Cartan tensor (characteristic of the Cartan connection). 
\par Straightforward computation, yields
\begin{align}\label{eq:Chris-symbs}
	\overstar{\Gamma}^k_{ij} - \Gamma^k_{ij} =  g^{kl}\left(C_{ijr}G^r_l - C_{ljr}G^r_i - C_{lir}G^r_j\right),
\end{align}
where $\Gamma_{ij}^k$ are the corresponding formal Christoffel symbols (of $g^{\y}$); see~\cite[(2.4.9)]{BCS}.
\par More explicitly, the Cartan connection (defined on the vertical bundle but more conveniently expressed on the pull-back bundle $\pi^*TM$) is given by
\begin{align*}
	\nabla_{{X}} \partial_i := \left( \overstar{\Gamma}^i_{jk} dx^j + C_{ij}^k\delta y^j \right)(X)  \; \partial_k, \quad X \in T\left(TM_0\right);
\end{align*}
in particular, 
$
\nabla_{\delta_j} \partial_i = \overstar{\Gamma}^k_{ij} \partial_k
$
and
$
\nabla_{\dt{\partial}_j} \partial_i = C_{ij}^k \partial_k
$.
For more details on the Cartan connection, see e.g.~\cite[Section 2.4.3]{AIM} and~\cite[Chapter II]{AZ}.
\par An important feature of the Cartan connection, which follows from its metric compatibility, is that
\begin{align}\label{eq:cart-met-comp}
	\left< U,V\right>^{\dt{}}_{\dt{c}} = \left< \nabla_{\dt{\widetilde{c}}}U,V\right>_{\dt{c}} +  \left< U, \nabla_{\dt{\widetilde{c}}} V\right>_{\dt{c}},
\end{align}
holds for any two vector fields $U$ and $V$ along $\gamma$; see~\cite{Mat2} and \cite[2.6.3]{AIM}, c.f.~\cite{SY}.  
\par A straightforward computation using
\begin{align*}
\dt{\widetilde{c}} = \dt{x}^l\partial_l + \ddt{x}^l\dt{\partial}_l = \dt{x}^l\delta_l + (\ddt{x}^l + 2G^l)\dt{\partial}_l,
\end{align*}
yields
\begin{align}\label{eq:cov-der}
	\nabla_{\dt{\widetilde{c}}} \partial_i  = \left( G^k_i + (\ddt{x}^l + 2G^l)C_{li}^k \right)\partial_k, \quad \text{and} \quad \nabla_{\dt{\widetilde{c}}} \dt{c} = (\ddt{x}^k + 2G^k)\partial_k.
\end{align}
Note that $\nabla_{\dt{\widetilde{c}}} \dt{c} = 0$ is the geodesic equation. 
\subsubsection{Berwaldian manifolds}
We say a Finslerian structure is Berwaldian when the Christoffel symbols $\overstar{\Gamma}^k_{ij}$ are independent of $\y$ alternatively if the Chern connection is the Levi-Civita connection of some Riemannian metric on $M$. 
\par Berwaldian manifolds are slightly more general than Riemannian structures; e.g. their geodesics coincide with the geodesics of a Riemannian metric. Though, they still comprise a large and interesting family of spaces. For more details, see~\cite[Chapter 10]{BCS}.
\subsubsection{Hessian}\label{subsubsec:hess}
\par The Finslerian Hessian of a smooth function $f$ is defined by
\begin{align*}
\Hess_F f (\y) := \nicefrac{d^2}{ds^2}\,\restr_{s=0} \left(f\circ {\eta_{\y}}\right),
\end{align*}
where $\eta_{\y}(s)$ is the geodesic with initial velocity $\y$. In local coordinates, 
\begin{align*}
\Hess_F(f)(\y) = y^iy^j f_{ij} - 2f_iG^i(\y);
\end{align*}
see~\cite[Chapter 14]{Shen-Lectures}. 
\par We will also encounter the horizontal Hessian, ${^{\sf horiz}\nabla}^2$, on \emph{scalar functions}, defined as follows:
\begin{align*}
{^{\sf horiz}\nabla}^2 f (X,Y) = XYf - \nabla_{\nabla_{X^h}Y}f = XY f - \nabla_{\left(\nabla_{X^h}Y\right)^h}f.
\end{align*}
\subsection{Curvature}
\phantom{}
\par \dunderline{1.2pt}{\small \textit{\textbf{\textsf{Notation:}}}} 
\begin{itemize}
	\item \textsl{$TM_0$ denotes the slit tangent bundle.}
	\smallskip
	\item \textsl{A vector field on $TM_0$ that projects to the vector field $X$ on $M$ under $\pi_*$ is denoted by $\widetilde{X}$.}
\end{itemize}
\par Three curvature tensors (forms) are assigned to $M$, given by
\begin{align*}
	\Riem^{\y}(X,Y) := \Omega\left(h\widetilde{X}, h\widetilde{Y}  \right), \quad
	P^{\y}(X, \dt{Y}) := \Omega\left(h\widetilde{X}, v\widetilde{Y}  \right), \quad
	Q^{\y}(\dt{X}, \dt{Y}) := \Omega(v\widetilde{X}, v\widetilde{Y}),
\end{align*}
where $X$ and $Y$ are vector fields and where $\Omega$ is the usual curvature form of the bundle $TTM_0$, i.e.,
\begin{align*}
\Omega(A,B) = [\nabla_{A},\nabla_B] - \nabla_{[A,B]}.
\end{align*}
Here, $\dt{X} := \mu(vX)$, where $\mu$ is given by 
$
\mu\left({\hat{X}}\right) = \nabla_{\hat{X}} (\y)
$.
Note that $\y$ is a canonical section of the pullback bundle $\pi^*(TM)$ (i.e., the pullback of the Liouville vector field); for details see~\cite{AZ}.
\par $\Riem^{\y}$ (also known as the $hh$-curvature) is the type of curvature tensor that we are concerned with in these notes.
\par A flag is a product $v\wedge w$ ($v,w$ are independent tangent vectors), with  $v$ referred to as the flag pole. The flag curvature $\K(v\wedge w)$ is defined as
\begin{align}\label{eq:flag}
	\K(v\wedge w) := \frac{\left< \Riem^{v}(v,w)w,v \right>_{v}}{\|v\|^2_{v}\|w\|^2_{w} - \left<v,w  \right>^2_v}.
\end{align}
For a $\y$-dependent orthonormal frame $\left\{e_i\right\}$, $\K(e_i\wedge e_j)$ are eigenvalues of the negative curvature operator $-\Riem$ acting on $\bigwedge^2(\mathcal{V}M)$.   
\subsection{Geodesic circles}\label{subsec:geod-circ}
\hfill
\par Using the Cartan connection, a geodesic circle in a Finslerian manifold is a smooth curve with constant first geodesic curvature and vanishing second geodesic curvature; for more details on curve theory in Finslerian manifolds, see~\cite[Chapters 2 and 4]{BF}. 
\par Let us list the well-known facts about geodesic circles in Finslerian manifolds that follow directly from standard ODE theory (e.g. see~\cite{SY} for details). For the definition of higher derivatives of $c$ see the next section.  
\begin{itemize}
	\item The curve $c(s)$ is a geodesic circle if and only if it satisfies the ODE,
	\begin{align*}
	c'''(s) + \|c''(s)\|^2c'(s) = 0, \quad \text{or equivalently,} \quad c'''(s) + \kappa_1^2 c'(s) = 0.
	\end{align*}
	\smallskip
	\item There is a unique geodesic circle parameterized by arc-length with the initial conditions $c(0) = p$, $c'(0) = u$, $c''(0) = v$ when $g_{u}(u,v) = 0$ and $\|u\| = 1$; in this case, $\kappa = \|c''(s)\| = \|v\|$.
\smallskip
	\item The curve $c(t)$ is a geodesic circle if and only if 
	\begin{align*}
	\dddt{c} - 3v^{-2}g_{\dt{c}}(\dt{c}, \ddt{c})\ddt{c} \equiv 0 \mod \dt{c}.
	\end{align*}
\end{itemize}
\subsection{Conformal Change of a Finslerian structure}\label{subsec:conf}
\par Consider a conformal change of the fundamental function,
$
\widebar{F}(x,{\y}) = e^{\sigma(x,{\y})}F(x,{\y})
$.
It is a classical result that $\sigma$ must be independent of $y$~\cite{Knebel}, and clearly  $\widebar{F} = e^{\sigma(x)}F$ is equivalent to $\widebar{g}^{\y} = e^{2\sigma(x)}g^{\y}$; see~\cite{Hach}.
\par Under the conformal change, the following holds:
\begin{align*}
	\widebar{y}_i = e^{2\sigma}y_i, \quad \widebar{C}_{ijk} = e^{2\sigma} C_{ijk}, \quad \text{and},\quad  \widebar{C}^{i}_{jk} = C^{i}_{jk}.
\end{align*}
Regarding the geodesic spray coefficients, one has
$
\widebar{G}^i = G^i - B^i
$
and
$
\widebar{G}^i_j = G^i_j - B^i_j,
$
where
\begin{align*}
	B^i = B^{ir}\sigma_r, \quad \text{where} \quad B^{ir} = 	\left( \nicefrac{F^2}{2} \right) (g^{ir} - 2\nicefrac{y^iy^r}{F^2} ) = (\nicefrac{F^2}{2})\;g^{ir} - y^ry^i.
\end{align*}
More explicitly,
\begin{align}\label{eq:con-dif-1}
	\widebar{G}^i - G^i = \left(y^ry^i - (\nicefrac{F^2}{2})\;g^{ir}\right)\sigma_r = y^i g(\grad^{\y} \sigma, \y) - (\nicefrac{F^2}{2})\sigma^i,
\end{align}
and nonlinear connection coefficients are related via
\begin{align}\label{eq:con-dif-2}
	\widebar{G}^i_j - G^i_j = \left(-y_jg^{ir} + F^2C^{ir}_j + \delta_j^ry^i + y^r\delta^i_j    \right)\sigma_r;
\end{align}
see~\cite{Hach} for details.
\section{The coordinate-free characterization}
\par The characterization of \cpt\!s in local coordinates is generalized to the Finslerian setting as follows.
\begin{theorem}[Local characterization, Shen-Yang~ {\cite{SY}}]\label{thrm:SY-local}
	A conformal transformation $\widebar{F} = e^{\sigma} F$ is a \cpt if and only if there exists a scalar function $\lambda(x)$ such that
	\begin{align}\label{eq:local-char}
		\rho_{_{i|j}} = \left(e^{-\sigma}\right)_{i|j} = \lambda g_{ij}, 
	\end{align}
	here, $_{|}$ denotes horizontal covariant derivative with respect to the Cartan or Chern connection.
\end{theorem}
\par Our first result is the derivation of the coordinate-free (Riccati-type) PDE  that characterizes a conformal \cpt serving as an analogue to the Riemannian equation \eqref{eq:pde-sigma}. 
\par \dunderline{1.2pt}{\small \textit{\textbf{\textsf{Notation:}}}}
\begin{itemize}
	\item \textsl{$c(s)$ is a unit-speed curve (w.r.t $F$) and $\left<\,\cdot\,, \,\cdot\, \right>_{c'}$ denotes the metric $g^{c'}$.}
	\smallskip
	\item We denote the natural lift of a curve $c(t)$ to the tangent bundle by $\widetilde{c}(t) := \left(c(t), \dt{c}(t) \right)$.
	\smallskip
	\item \textsl{ $UM$ refers to the indicatrix bundle, i.e., the set of all vectors $X$ with $F(X) = 1$.}
	\smallskip
	\item $\kappa$ and $\widebar{\kappa}$ are the geodesic curvatures computed w.r.t. $F$ and  $\widebar{F} = e^{\sigma} F$ respectively. 
	\smallskip
	\item  The higher derivatives of $c$ are defined as $c^{(k)}(t) := \nabla^{k-1}_{\dt{\widetilde{c}}} \dt{c}$. For the covariant derivatives of $c$ w.r.t. a general parameter and the arc-length parameter (w.r.t. $g_{\dt{c}}$), we will adopt the standard notation $^{\dt{}}:= \nabla_{\dt{\widetilde{c}}(t)}$ and $':= \nabla_{{\widetilde{c}}'(s)}$ respectively. However, for better readability, all derivatives of the coordinate functions will always be dotted.
	\smallskip
	\item  For the derivative (still using the Cartan connection of $g$) w.r.t. the arc-length parameter $\widebar{s}$ of a conformal change $\widebar{g}_{\dt{c}} = e^{2\sigma}g_{\dt{c}}$, we adopt the notation $^{\ast}:= \nabla_{\widetilde{c}^\ast}$.
\end{itemize} 
\begin{lemma}\label{lem:key-1}
	The following identities hold along $c(s)$.
	\begin{enumerate}
		\item \label{lemma:item1}
		\makebox[\linewidth]{\(\begin{aligned}[t]
				\left(\widebar{\nabla}_{\widetilde{c}'} - \nabla_{\widetilde{c}'}\right) c' = - \grad^{c'}\sigma + 2\sigma'c';
			\end{aligned}\)}
		\smallskip
		\item\label{lemma:item2} 
		\makebox[\linewidth]{\(\begin{aligned}[t]
				\left( \widebar{\nabla}_{\widetilde{c}'} - \nabla_{\widetilde{c}'} \right)	\nabla_{\widetilde{c}'}c' = \left<\grad^{c'} \sigma, c''\right>c' + \sigma'c'';
			\end{aligned}\)}
		\smallskip
		\item \label{lemma:item3}
		\makebox[\linewidth]{\(\begin{aligned}[t]
				\left(\widebar{\nabla}_{\widetilde{c}'} - {\nabla}_{\widetilde{c}'} \right)\grad^{c'} \sigma = \|\grad^{c'}\sigma\|^2;
			\end{aligned}\)}
		\smallskip
		\item \label{lemma:item4}
		\makebox[\linewidth]{\(\begin{aligned}[t]
				\left<\widebar{\nabla}_{\widetilde{c}'}c', c'\right> = \sigma'.
			\end{aligned}\)}	
	\end{enumerate}
\end{lemma}
\begin{proof}[\footnotesize \textbf{Proof}]
	\hfill
	\begin{enumerate}
		\item [(\ref{lemma:item1}).] Using \eqref{eq:cov-der}, one has
		\begin{align*}
			\left(\widebar{\nabla}_{\widetilde{c}'} - \nabla_{\widetilde{c}'}\right) c' &= 2 \left(\widebar{G}^k - G^k \right)\partial_k\\
			&= - 2B^k\partial_k \\
			&= \left(- g^{kr}  + 2 \dt{x}^r\dt{x}^k  \right)\sigma_r\partial_k\\
			&= -\grad^{c'} \sigma + 2\left<c', \grad^{c'} \sigma\right>c'.
		\end{align*}
		\smallskip
		\item [(\ref{lemma:item2}).] 
		From \eqref{eq:cov-der} and in combination with \eqref{eq:con-dif-1} and \eqref{eq:con-dif-2}, we have
		\begin{align}\label{eq:difference-con-coef}
			\left( \widebar{\nabla}_{\widetilde{c}'} - \nabla_{\widetilde{c}'} \right) \partial_j &= \left(\widebar{G}^i_j - G^i_j + 2C^{i}_{jl} (\widebar{G}^l - G^l)\right)\partial_i \notag\\
			&=\left(-y_jg^{ir} + \delta_j^ry^i + y^r\delta^i_j    \right)\sigma_r\partial_i \notag\\
			&= - y_j\sigma^i\partial_i + (\sigma_j)y^i\partial_i +  \sigma_ry^r\partial_j,
		\end{align}
		in which $y^i = \dt{x}^i$; thus,
		\begin{align*}
			\left( \widebar{\nabla}_{\widetilde{c}'} - \nabla_{\widetilde{c}'} \right) \partial_j = \left<\grad^{c'}\sigma,\partial_j\right>c'- \left<c', \partial_j\right>\grad^{c'} \sigma + \left<\grad^{c'} \sigma, c'\right>\partial_j.
		\end{align*}
		Therefore, by tensoriality of the difference of covariant derivatives, we deduce that
		\begin{align*}
			\left( \widebar{\nabla}_{\widetilde{c}'} - \nabla_{\widetilde{c}'} \right)	\nabla_{\widetilde{c}'}c' 
			&= - \left<\nabla_{\widetilde{c}'}c', c'\right>\grad^{c'} \sigma + \left<\grad^{c'} \sigma, \nabla_{\widetilde{c}'}c'\right>c' + \left<\grad^{c'}\sigma, c'\right>\nabla_{\widetilde{c}'}c'\\
			&= \left<\grad^{c'} \sigma, \nabla_{\widetilde{c}'}c'\right>c' + \left<\grad^{c'}\sigma, c'\right>\nabla_{\widetilde{c}'}c';
		\end{align*}
		and in particular,
		\begin{align*}
			\left( \widebar{\nabla}_{\widetilde{c}'} - \nabla_{\widetilde{c}'} \right)	\nabla_{\widetilde{c}'}c'  \equiv  \left<\grad^{c'}\sigma, c'\right>\nabla_{\widetilde{c}'}c' \mod c'.
		\end{align*}
		Here, we have used $\left<\nabla_{\widetilde{c}'}c', c'\right> = 0$ which follows from the metric compatibility \eqref{eq:cart-met-comp}.
		\smallskip
		\item [(\ref{lemma:item3}).] 
		We first note that by metric compatibility, we have $^{\sf vert}\nabla g = 0$. Since $^{\sf vert}\nabla$ is a linear connection, we deduce that $^{\sf vert}\nabla g^{-1} = 0$ as well. Applying this to $\bar{g}$ we get
		\begin{align*}
		 e^{-2\sigma}\; {^{\sf vert}\widebar{\nabla}} \left( g^{-1}  \right)	 = \widebar{\nabla} \left( e^{-2\sigma} g^{-1}  \right) = \widebar{\nabla} \left( \bar{g}^{-1}  \right) = 0;
		\end{align*}
		therefore, ${^{\sf vert}\widebar{\nabla}} \left( g^{-1}  \right) = 0$. In particular,
		\begin{align*}
		(\widebar{\nabla}_{\widetilde{c}'} - \nabla_{\widetilde{c}'}) \sigma^i &=  (\widebar{\nabla}_{\widetilde{c}'} - \nabla_{\widetilde{c}'}) (g^{ir}\sigma_r)\\
		&= 0.
		\end{align*}		
		Now using \eqref{eq:difference-con-coef}, we obtain (with $y^i = \dt{x}^i$)
		\begin{align*}
			(\widebar{\nabla}_{\widetilde{c}'} - \nabla_{\widetilde{c}'})\grad^{c'} \sigma &= \sigma^j (\widebar{\nabla}_{\widetilde{c}'} - \nabla_{\widetilde{c}'})\partial_j\\
			 &= \sigma^j( - y_j\sigma^i\partial_i + (\sigma_j)y^i\partial_i +  \sigma_ry^r\partial_j) \\
			&= \|\grad^{c'} \sigma\|^2c' \equiv 0 \quad \mod c'.
		\end{align*}
		\smallskip
		\item [(\ref{lemma:item4}).]
		\begin{align*}
			\left<\widebar{\nabla}_{\widetilde{c}'}c', c'\right> &= \left<\nabla_{\widetilde{c}'}c' -\grad^{c'} \sigma + 2\left<c', \grad^{c'} \sigma\right>c' ,c'\right> \\
			&= \left<\nabla_{\widetilde{c}'}c', c'\right> + \left<\grad^{c'} \sigma, c'\right> \\
			&= \left<\grad^{c'} \sigma, c'\right> .
		\end{align*}
	\end{enumerate}
\end{proof}
\begin{theorem}[Circle-wise characterization]\label{thrm:path-wise}
	$\bar{F} = e^{\sigma} F$ is a local regular \cpt if and only if $\sigma$ satisfies 
	\begin{align}\label{eq:main-3}
		\nabla_{\widetilde{c}'} (\grad^{c'} \sigma) - g(c', \grad^{c'}\sigma)\grad^{c'}\sigma = f(x,c')c',
	\end{align}
	with
	\begin{align}\label{eq:f-form}
	 f(x,c') = \sigma'' + \left< c'', \grad^{c'}\sigma \right> - \|\grad^{c'} \sigma\|^2 + e^{2\sigma}\widebar{\kappa}^2 -  \kappa^2.
	 \end{align}
	 Alternatively,
	 \begin{align*}
	 	f(x,c') = \sigma'' - \left< c'', \grad^{c'}\sigma \right> - \left(\sigma'\right)^2.
	 \end{align*}
\end{theorem}
\begin{proof}[\footnotesize \textbf{Proof}]
\par Let $c(s)$ be a geodesic circle with arc-length parameter (w.r.t. $g$). From Lemmas~\ref{lem:key-1}, we deduce that
	\begin{align}\label{eq:main-1}
		\widebar{\nabla}_{\widetilde{c}'}\widebar{\nabla}_{\widetilde{c}'} c' - \nabla_{\widetilde{c}'}\nabla_{\widetilde{c}'} c' &= \widebar{\nabla}_{\widetilde{c}'} 	\left(\widebar{\nabla}_{\widetilde{c}'} - \nabla_{\widetilde{c}'}\right) c' + \left( \widebar{\nabla}_{\widetilde{c}'} - \nabla_{\widetilde{c}'} \right)	\nabla_{\widetilde{c}'} c' \notag \\
		&= \widebar{\nabla}_{\widetilde{c}'}(- \grad^{c'}\sigma + 2\sigma'c') + \left<\grad^{c'}\sigma, c''\right>c' + \sigma'c'' \notag\\
		&=  - {\nabla}_{{\widetilde{c}}'} (\grad^{c'} \sigma) - \|\grad^{c'}\sigma\|^2c'+ 2\sigma'\left( c'' - \grad^{c'}\sigma + 2\sigma'c' \right) + 2\sigma''c'\\
		& \phantom{sajjad} + \left<\grad^{c'}\sigma, c''\right>c' + \sigma'c''\notag\\
		&= - \nabla_{\widetilde{c}'} \grad^{c'}\sigma - 2\sigma'\grad^{c'}\sigma + 3 {\sigma}'c''  \notag\\
		&\phantom{sajjad} + \Big(4(\sigma')^2 +  3\sigma'' - \left<c', \nabla_{\widetilde{c}'} \grad^{c'} \sigma\right> - \|\grad^{c'} \sigma\|^2\Big)c'\notag.
	\end{align}
	Thus,
	\begin{align}\label{eq:concirc}
		&\widebar{\nabla}_{\widetilde{c}'}\widebar{\nabla}_{\widetilde{c}'} c' - \nabla_{\widetilde{c}'}\nabla_{\widetilde{c}'} c' - 3 \left<\widebar{\nabla}_{\widetilde{c}'}c', c'\right>\widebar{\nabla}_{\widetilde{c}'}c'\\ 
		&= - \nabla_{\widetilde{c}'} \grad^{c'}\sigma - 2\sigma'\grad^{c'}\sigma + 3 {\sigma}'c''  \notag\\
		&\phantom{sajjad} + \Big(4(\sigma')^2 +  3\sigma'' - \left<c', \nabla_{\widetilde{c}'} \grad^{c'} \sigma \right> - \|\grad^{c'} \sigma\|^2\Big)c'\notag\\
		&\phantom{sajjad} - 3\sigma'\left( c'' - \grad^{c'}\sigma + 2\sigma'c' \right)\notag\\
		&= - \nabla_{\widetilde{c}'} \grad^{c'}\sigma + \sigma'\grad^{c'}\sigma \notag\\
		&\phantom{sajjad} + \Big(3\sigma''-2(\sigma')^2- \left<c', \nabla_{\widetilde{c}'} \grad^{c'} \sigma\right> - \|\grad^{c'} \sigma\|^2\Big)c'\notag .
	\end{align}
	\par By the characterization of geodesic circles (with arbitrary parameters) in terms of tangency relations (see~\textsection\thinspace\ref{subsec:geod-circ}), we know that for a \cpt, the LHS in \eqref{eq:concirc} is tangent to $c'$. Therefore, one concludes that
	\begin{align*}
		\nabla_{\widetilde{c}'} (\grad^{c'} \sigma) - g(c', \grad^{c'}\sigma)\grad^{c'}\sigma \equiv 0 \mod c'.
	\end{align*}
\par On the other hand, upon applying the standard change of parameter formula,
	\begin{align*}
		\dddt{c} - 3v^{-2}g_{\dt{c}}(\dt{c}, \ddt{c})\ddt{c} = v^3c''' + \Big( \left( v^{-1}g_{\dt{c}}(\dt{c}, \ddt{c}) \right)^{\dt{}} - 3v^{-3}g_{\dt{c}}\left(\dt{c}, \ddt{c}\right)^2 \Big) c',
	\end{align*}
(easily verified by basic computation or see~\cite{Kuh}), to the metric $\widebar{g}$, noting
	\begin{align*}
		v = \|c'\|_{\widebar{g}} = e^{\sigma}\|c'\|_g = e^\sigma,
	\end{align*}
	we obtain
	\begin{align*}
		\widebar{\nabla}_{\widetilde{c}'}\widebar{\nabla}_{\widetilde{c}'} c'- 3 \left<\widebar{\nabla}_{\widetilde{c}'}c', c'\right>\widebar{\nabla}_{\widetilde{c}'}c'&=
		\widebar{\nabla}_{\widetilde{c}'}\widebar{\nabla}_{\widetilde{c}'} c'- 3 \frac{\widebar{g}(\widebar{\nabla}_{\widetilde{c}'}c', c')}{\widebar{g}\left(c',c' \right)}\widebar{\nabla}_{\widetilde{c}'}c' \\&= e^{3\sigma}c^{\ast\ast\ast} + \Big( \left( e^{-\sigma}\widebar{g}({c}', \widebar{\nabla}_{\widetilde{c}'} c') \right)' - 3e^{-3\sigma}\widebar{g}\left({c}', \widebar{\nabla}_{\widetilde{c}'} c'\right)^2 \Big) c^\ast.
	\end{align*}
	According to item (\ref{lemma:item4}) in Lemma~\ref{lem:key-1}, we get
	\begin{align*}
		\widebar{g}\left({c}', \widebar{\nabla}_{\widetilde{c}'} c'\right) = e^{2\sigma} \left<c',\widebar{\nabla}_{\widetilde{c}'} c'\right> = e^{2\sigma}\sigma';
	\end{align*}
	hence,
	\begin{align*}
		\left( e^{-\sigma}\widebar{g}({c}', \widebar{\nabla}_{\widetilde{c}'} c') \right)' = e^{\sigma}\left( (\sigma')^2 + \sigma''\right).
	\end{align*}
	We note that,
	\begin{align*}
		c^\ast = e^{-\sigma}c', \quad \text{and}, \quad c^{\ast\ast\ast} = -\widebar{\kappa}^2 c^{\ast},
	\end{align*}
	and these give us
	\begin{align*}
		\widebar{\nabla}_{\widetilde{c}'}\widebar{\nabla}_{\widetilde{c}'} c'- 3 g(\widebar{\nabla}_{\widetilde{c}'}c', c')\widebar{\nabla}_{\widetilde{c}'}c'&=
		e^{3\sigma}c^{\ast\ast\ast} + e^{\sigma}\Big( \sigma'' - 2(\sigma')^2  \Big) c^\ast\\
		&= \Big( - e^{2\sigma} \widebar{\kappa}^2 + \sigma'' - 2(\sigma')^2 \Big) c'.
	\end{align*}
	Consequently,
	\begin{align}\label{eq:2nd-form}
		\widebar{\nabla}_{\widetilde{c}'}\widebar{\nabla}_{\widetilde{c}'} c' - \nabla_{\widetilde{c}'}\nabla_{\widetilde{c}'} c' - 3 g(\widebar{\nabla}_{\widetilde{c}'}c', c')\widebar{\nabla}_{\widetilde{c}'}c'= \left( \kappa^2 - e^{2\sigma}\widebar{\kappa}^2 + \sigma'' - 2(\sigma')^2 \right)c'.
	\end{align}
	Upon combining \eqref{eq:concirc} and \eqref{eq:2nd-form} we get
	\begin{align*}
		&	- \nabla_{\widetilde{c}'} \grad^{c'}\sigma + \sigma'\grad^{c'}\sigma + \Big(3\sigma''-2(\sigma')^2- \left<c', \nabla_{\widetilde{c}'} \grad^{c'} \sigma\right> - \|\grad^{c'} \sigma\|^2\Big)c' \\
		& = \left( \kappa^2 - e^{2\sigma}\widebar{\kappa}^2 + \sigma'' - 2(\sigma')^2 \right)c',
	\end{align*}
	or
	\begin{align}\label{eq:along-c}
		\nabla_{\widetilde{c}'} \grad^{c'}\sigma - \sigma'\grad^{c'}\sigma = \Big(2\sigma''- \left<c', \nabla_{\widetilde{c}'} \grad^{c'} \sigma\right> - \|\grad^{c'} \sigma\|^2 + e^{2\sigma}\widebar{\kappa}^2 -  \kappa^2 \Big)c'.
	\end{align}
	Therefore,
	\begin{align*}
		f(x,c') &= 2\sigma''- \left<c', \nabla_{\widetilde{c}'} \grad^{c'} \sigma\right> - \|\grad^{c'} \sigma\|^2 + e^{2\sigma}\widebar{\kappa}^2 -  \kappa^2\\
		&= \sigma'' + \left< c'', \grad^{c'}\sigma \right> - \|\grad^{c'} \sigma\|^2 + e^{2\sigma}\widebar{\kappa}^2 -  \kappa^2.
	\end{align*}
\par The second expression for $f$ is easily obtained by contracting both sides with $c'$ and using metric compatibility. 
\end{proof}
\dunderline{1.2pt}{\small \textit{\textbf{\textsf{Notation:}}}}
\textsl{As is customary in hypersurface theory, when $u^1$ is the normal coordinate to a family of hypersurfaces, we signify the indices $2,\cdots,n$ (corresponding to hypersurface coordinates $u^2,\cdots,u^n$) by small case Greek letters, $\alpha,\beta,\cdots$.}
\begin{theorem}[Geometry of local regular solutions-I]\label{thrm:geom-1}
	Suppose $\sigma$ satisfies \eqref{eq:main-3} along geodesic circles. Around a regular point $x$ of $\sigma$, the following properties hold for the (regular) level set $\Sigma_x:= \sigma^{-1}(\sigma(x))$.
	\begin{enumerate}
		\item \label{path-char:item1} The only unit-speed forward geodesics that are mapped to geodesics are those that solve
		$
		c'(s) =  \pm {\sf n}_{c(s)}
		$.
		\par Consequently, passing through each regular point, there is precisely one geodesic that is mapped to a geodesic. The integral curves of the field ${\sf n}$ will be called normal geodesics.
		\smallskip
		\item \label{path-char:item2} There exist local coordinates $(u^1 = u,u^2, \cdots, u^n)$ where $u^2, \cdots, u^n$ form coordinates for $\Sigma_x$ and $u$ is the arc-length of the normal geodesics i.e. the coordinate $u$ is the one that integrates 
		\begin{align*}
			\partial_u := \nicefrac{\grad^{\y} \sigma}{F(\grad^{\y} \sigma)} = \nicefrac{\grad \sigma}{F(\grad \sigma)};
		\end{align*}
		see~\textsection\thinspace\ref{subsubsec:grad}.
		In these coordinates, $f$ takes the form
		\begin{align*}
			f = \sigma_{uu} + \sigma_u^2 = -\nicefrac{\rho''}{\rho}.
		\end{align*}
		In particular, the function $f$ is a \underline{scalar function} that is locally constant on  $\Sigma_x$; thus, locally $f$ is a smooth function of $\sigma$.
		\smallskip
		\item \label{path-char:item3} $g_{1\alpha}(\y) \equiv 0$ and $g_{_{11}}({\sf n}) = \|{\sf n}\|_{\sf n} = F({\sf n}) \equiv 1$;  as a result, $F^2(\grad \sigma)$ is locally constant along $\Sigma$
		\smallskip
		\item \label{path-char:item4} 
		\makebox[\linewidth]{\(\begin{aligned}[t]
				\nabla_{\dt{\partial}_j} \grad^{c'}\sigma = 0.
			\end{aligned}\)}
		\smallskip
		\item \label{path-char:item5} In the coordinates given above, we have
		\begin{enumerate} 
			\smallskip
			\item \label{path-char:item5-(a)} $C_{111} = C_{11\alpha} = C_{1\alpha\beta} = 0$, consequently, $g_{_{11}} \equiv 1$ and $\grad^{\y}\sigma$ is independent of $\y$. Furtheremore, $F(- {\sf n}) = 1$.
			\smallskip
			\item \label{path-char:item5-(b)} $G^s(\sf n) = 0$ for all $1\le s \le n$. 
			\smallskip
			\item  \label{path-char:item5-(c)} $\dt{\partial}_1 G^i \equiv 0$ i.e. $G^i_1\equiv 0$ and consequently, $\dt{\partial}_1 G^i_j \equiv 0$.
		\end{enumerate}
	\end{enumerate}
\end{theorem}
\begin{proof}[\footnotesize \textbf{Proof}]
	\emph{Throughout this proof, $X$ is assumed to be tangent to the regular level set $\Sigma_x$}. 
	\par First, an elementary observation. By the definition of gradients in~\textsection\thinspace\ref{subsubsec:grad}, we know that along a unit-speed curve $c(s)$, the inequality $\|\grad^{c'}\sigma\|_{c'} \ge \sigma'$ holds, with the equality if and only if 
	\begin{align*}
		c' = F(\grad^{c'}\sigma)^{-1}\grad^{c'}\sigma.
	\end{align*}
	This, in particular, implies
	\begin{align*}
		F(\grad^{c'}\sigma) = g_{c'}(\grad^{c'}\sigma,\grad^{c'}\sigma)^{\nicefrac{1}{2}}.
	\end{align*}
	\par As a result,
	$
	\|\grad^{c'}\sigma\|^2 = (\sigma')^2$
	if and only if
	\begin{align*}
		\pm c' = \nicefrac{\grad^{c'}\sigma}{F(\grad^{c'}\sigma)} = \nicefrac{\grad \sigma}{\|\grad \sigma\|_{\grad \sigma}} = \nicefrac{\grad \sigma}{\|\grad \sigma\|_{c'}}.
	\end{align*}
	This means that such a $c'$ solving the above relation must be $\pm \nicefrac{\grad \sigma}{F(\grad \sigma)} = \pm {\sf n}$. Note the nontrivial implicit fact here, that $F(-{\sf n}) = 1$; this will be established in item (\ref{path-char:item3}). 
\begin{enumerate}
		\item []{\bf{\footnotesize \textsf{Proof of (\ref{path-char:item1})}}}. Contracting both sides of \eqref{eq:along-c} with $c'$, we deduce that
		\begin{align*}
			\left<\nabla_{\widetilde{c}'} \grad^{c'}\sigma, c'\right> - (\sigma')^2 = 2\sigma''- \left<c', \nabla_{\widetilde{c}'} \grad^{c'} \sigma\right> - \|\grad^{c'} \sigma\|^2 + e^{2\sigma}\widebar{\kappa}^2 -  \kappa^2,
		\end{align*}
		yielding
		\begin{align*}
			2\sigma''- 2\left<c', \nabla_{\widetilde{c}'} \grad^{c'} \sigma\right> - \|\grad^{c'} \sigma\|^2 + e^{2\sigma}\widebar{\kappa}^2 -  \kappa^2 + (\sigma')^2 = 0.
		\end{align*}
		Using metric compatibility, we get
		\begin{align*}
			2 \left<\grad^{c'}\sigma, c''  \right> - \|\grad^{c'} \sigma\|^2 + e^{2\sigma}\widebar{\kappa}^2 -  \kappa^2 + (\sigma')^2 = 0.
		\end{align*}
		Thus, we deduce that along a geodesic $c$,
		\begin{align*}
			\|\grad^{c'} \sigma\|^2  - (\sigma')^2 = e^{2\sigma}\widebar{\kappa}^2.
		\end{align*}
		Therefore, if $c$ is a geodesic for $\widebar{g}$ as well, the it solves
		$
		\pm c' = \nicefrac{\grad^{c'} \sigma}{F(\grad^{c'}\sigma)},
		$
		which, by the argument at the beginning of this section, means that $c'$ must be $\pm {\sf n}$.
		\par Next, we show that the normal curves $\gamma_{\sf n}(s)$ with tangent ${\sf n}$ indeed are geodesics. By the \eqref{eq:main-3}, we obtain
		\begin{align}\label{eq:pde-normal-dir}
			\nabla_{\widetilde{\gamma_{\sf n}}'} \grad\sigma - d\sigma({\sf n})\grad \sigma = f{\sf n},
		\end{align}
		which implies 
		$
		n \parallel \nabla_{\widetilde{\gamma_{\sf n}}'} {\sf n}.
		$
		On the other hand, since $\|{\sf n}\| = 1$ and by metric compatibility, we know that
		$
		\nabla_{\widetilde{\gamma_{\sf n}}'} {\sf n} \perp {\sf n}.
		$
		Hence,
		$
		\nabla_{\widetilde{\gamma_{\sf n}}'} {\sf n} = 0
		$
		which means that (the forward) integral curves of the field ${\sf n}$ are unit-speed geodesics. 
		\par In contrast, the forward integral curves of $-{\sf n}$ are geodesics only when the geodesic in the direction ${\sf n}$ is reversible. This is indeed the case since we have $\widetilde{\gamma_{-{\sf n}}}' = - \widetilde{\gamma_{{\sf n}}}'$ both sides of \eqref{eq:pde-normal-dir} are linear in ${\sf n}$ and by noting $F(-{\sf n}) = 1$ as will be shown in item (\ref{path-char:item3}) below.
		\par Therefore, the only forward geodesics that are mapped to geodesics are those in the direction of $\pm{\sf n}$. Hence, as sets, there is only one geodesic curve (as point sets) that is mapped to a geodesic. Along the forward geodesic in the direction $\sf n$, $\sigma$ is increasing. 
		\smallskip
		\item []{\bf{\footnotesize \textsf{Proof of (\ref{path-char:item2})}}}. 
		\par Take $u^1 = u$ to be the forward arc-length parameter along the normal curves. Choose any coordinate system $u^2,\cdots, u^{n-1}$ on the regular hypersurface $\sigma^{-1}(c)$. Extend these coordinates to a neighborhood of the point by keeping them to be constant along the normal unit-speed geodesic curves. 
		\par Obviously, the distances between $u$-levels are determined by the differences in the values of $u$. Clearly, $u$-levels correspond to the same hypersurface ($\sigma$-levels). Therefore, $\sigma$ is a function of $u$ and we write $\sigma(u)$. Clearly
		$
		g^{\sf n}(\partial_1, \partial_1) = 1
		$
		and
		$
		g^{\sf n}(\partial_1, \partial_{\alpha}) = 0
		$
		i.e. $\partial_u = \partial_1 = {\sf n}$, but we continue to write them separately. 
		\par Since normal curves are geodesics, one has 
		\begin{align*}
			f\partial_1 = \nabla_{\widetilde{\gamma_{\sf n}}'} \grad \sigma - (\sigma_1)\grad \sigma =\nabla_{\widetilde{\gamma_{\sf n}}'}\left(\sigma_1 \partial_1\right) -  \left(\sigma_1\right)^2\partial_1 = \left(\sigma_{11} - \left(\sigma_1\right)^2\right) \partial_1;
		\end{align*}
		thus,
		\begin{align*}
			f(u) = \sigma_{11} - (\sigma_1)^2.
		\end{align*}
		\smallskip
			\item []{\bf{\footnotesize \textsf{Proof of (\ref{path-char:item3})}}}. 
		\par Since $\grad^{\y} \sigma$ is perpendicular to $\Sigma$ w.r.t. $g^{\y}$, we deduce that $g^{\y}(\partial_1,\partial_\alpha) = 0$ for all $\y$. So, $g_{1\alpha} \equiv 0$. Consequently, from the block form of $(g_{ij})$, we also deduce that $g^{1\alpha} \equiv 0$.
		\smallskip
		\item []{\bf{\footnotesize \textsf{Proof of (\ref{path-char:item4})}}}. 
	 \begin{align*}
	 	\nabla_{\dt{\partial}_j} \grad^{c'}\sigma &= \nabla_{\dt{\partial}_j} (\sigma^r\partial_r) \\
	 	&= \dt{\partial}_j(g^{il}\sigma_l)\partial_i + \sigma^r C_{rj}^k \partial_k \\
	 	&= \dt{\partial}_j(g^{il})\sigma_1\delta_{i1}\partial_1\\
	 	&= 0,
	 \end{align*}
	in which we have used $y^kC^i_{jk} = 0$ and $g^{1l} = \delta^{1l}$.
	\smallskip
	\item []{\bf{\footnotesize \textsf{Proof of (\ref{path-char:item5})}}}.
\begin{enumerate}
		\item []{{\footnotesize \textsf{Proof of (\ref{path-char:item5-(a)})}}}. Clearly from $g_{1\alpha}\equiv 0$, one deduces that $C_{1i\alpha} \equiv 0$; in particular, $C_{11\alpha} \equiv 0$. Due to the block form of $(g_{ij})$, we deduce that $g^{1\alpha} \equiv 0$.
		\par Utilizing $C_{1i\alpha} \equiv 0$ along with the positive $0$-homogeneity of $g$, we get
		\begin{align*}
		g_{11}(\y)  = g_{11}({\sf n}),\quad \text{when}, \quad y^1 > 0,
		\end{align*}
		and
		\begin{align*}
		g_{11}(\y)  = g_{11}(- {\sf n}),\quad \text{when}, \quad y^1 < 0.
		\end{align*}
		Indeed, this follows by considering the smooth path
		\begin{align*}
			(y^1,{\text{$\small \boldsymbol{0}$}}) + t(0, y^2,\cdots, y^n), \quad 0\le t\le 1
		\end{align*} 
		along which, $g_{11}$ is constant. 
		\par By continuity, this implies $g_{11}(\y) \equiv  g_{11}(-{\sf n}) \equiv g_{11}({\sf n}) = 1$ for all $\y \neq 0$; in particular, $C_{111} \equiv 0$.
		\smallskip 
		\item []{{\footnotesize \textsf{Proof of (\ref{path-char:item5-(b)})}}}. The unit normal geodesics, in coordinates, are given by $x^{i}(s) = \delta^{i0}s$ (after perhaps adjusting by a constant shift). Using the geodesic ODE in local coordinates (see \eqref{eq:cov-der}), we deduce that $G^i(x,{\sf n}) = 0$ for all $i$ on the regular set of $\sigma$.
		\smallskip
		\item []{{\footnotesize \textsf{Proof of (\ref{path-char:item5-(c)})}}}. Note, by item (\ref{path-char:item5-(a)}) above, that
		\begin{align*}
		F^2(\y) = \left(y^1 \right)^2 + g_{\alpha\beta}y^\alpha y^\beta;
		\end{align*}
		therefore, by item (\ref{path-char:item5-(a)}) above, 
		\begin{align*}
		\dt{\partial}_1 F^2 = 2y^1 + C_{1\alpha\beta}y^\alpha y^\beta = 2y^1.
		\end{align*}
		As a result,
		\begin{align*}
		\partial_j \dt{\partial}_1 F^2 = 0, \quad \forall j;
		\end{align*}
	 substituting this in $\dt{\partial}_1G^i$ yields the desired result. 
	\end{enumerate}
	\end{enumerate}
\end{proof}
\begin{theorem}[Coordinate-free characterization]\label{thrm:coord-free}
	$\bar{F} = e^{\sigma} F$ is a local regular \cpt if and only if $\sigma$ satisfies the coordinate-free geometric PDE
	\begin{align}\label{eq:Fins-main-1}
		\nabla_{X^h} (\grad \sigma) -  d\sigma(X)\grad\sigma = f\left(x\right)X, \quad \forall X \in T_xM.
	\end{align}
	or equivalently, 
	\begin{align}\label{eq:Fins-cpt-hess}
		\Hess_F \sigma - d\sigma\otimes d\sigma \circ \ident\times\ident = f\ident, \quad \text{on} \quad SM,
	\end{align}
	or,
	\begin{align}\label{eq:Fins-cpt-horiz-hess}
		^{\sf horiz}\nabla^2 \sigma - d\sigma\otimes d\sigma = f g.
	\end{align}
	Furtheremore, $f$ is the scalar function (i.e., only depending on $x$),
	\begin{align}\label{eq:f-simple}
		f(x) = \sigma''\restr_{\nicefrac{\y}{F(\y)}} - F(\grad \sigma)^2 + e^{2\sigma}\widebar{\kappa}^2,
	\end{align}
	where, $\sigma''\restr_{\nicefrac{\y}{F(\y)}}$ denotes the second derivative in the direction of the unit-speed geodesic with initial velocity $\nicefrac{\y}{F(\y)}$ (indeed the Finslerian Hessian, $Hess_F \sigma (\nicefrac{\y}{F(\y)})$) and $\widebar{\kappa}$ is the curvature of this geodesic w.r.t. $\widebar{F}$. 
\end{theorem}
\begin{proof}[\footnotesize \textbf{Proof}]
	\par Based on item (\ref{path-char:item4}) of Theorem~\ref{thrm:geom-1}, we get $\nabla_{X^v} \grad^{c'}\sigma = 0$ for any $X \in TM$. Consequently, \eqref{eq:main-3} yields
	\begin{align*}
		\nabla_{\dt{x}^l\delta_l} (\grad^{c'} \sigma) - g(\dt{x}^l\partial_l, \grad^{c'}\sigma)\grad^{c'}\sigma \equiv 0 \mod c'.
	\end{align*}
	Note that $\dt{x}^l\delta_l = h\left(\widetilde{c}'\right) = (c')^h$.
	\par By virtue of the existence result on geodesic circles with any given initial tangent $X$ with $\|X\|=1$  (\textsection\thinspace\ref{subsec:geod-circ}), we get \eqref{eq:Fins-main-1}; noting that the fact that $f$ is a scalar function is established in item (\ref{path-char:item2}) of Theorem~\ref{thrm:geom-1}. 
\par By item (\ref{path-char:item5}) in Theorem~\ref{thrm:geom-1}, $\grad^{\y} \sigma$ is indeed independent of $\y$; consequently, we deduce that
	\begin{align*}
		\nabla_{X^h} (\grad \sigma) - g^{\y}(X, \grad\sigma)\grad\sigma = f(x)X, \quad \forall X \in T_xM,
	\end{align*}
	holds for every $x$ in the regular set of $\sigma$. 
\par Recall that by the $h$-symmetry of the Cartan connection~\cite[Section 2.4.3]{AIM}, the iterated derivative $^{\sf horiz}\nabla^{2}_{X,Y}$ (the linear Hessian) on vector fields is symmetric, thus, defines a symmetric tensor; in particular, only depends on the pointwise values of $X$ and $Y$.  
	\par Contracting both sides of \eqref{eq:Fins-main-1} with the vector field $Y$, we obtain
	\begin{align*}
		\left<\nabla_{X^h} (\grad \sigma), Y\right>_{\y} - d\sigma(X)d\sigma(Y) = f(x)\left< X,Y \right>_{\y},
	\end{align*}
	which gives
	\begin{align*}
		\nabla_{X^h} \left(d\sigma(Y)\right) - d\sigma\left( \nabla_{X^h} Y \right) - d\sigma(X)d\sigma(Y) = f(x)\left< X,Y \right>_{\y}.
	\end{align*}
	This is nothing but
	\begin{align*}
		{^{\sf horiz}\nabla^2\sigma}(X,Y) - d\sigma^2(X,Y) = fg_{\y}(X,Y).
	\end{align*}
	Based on the definition of $\Hess_F$ (\textsection\thinspace\ref{subsubsec:hess}), \eqref{eq:Fins-cpt-hess} is a special case of \eqref{eq:Fins-cpt-horiz-hess} by setting $\y = X = Y = c'$ along the geodesic $c(s)$.
\end{proof}
\begin{remark}
	\textsl{The general form of the coordinate-free characterization can indeed be inferred from the local characterization \eqref{eq:local-char} of Shen-Yang; see~\cite[(80)]{SY}. However, our approach above has the added benefit of providing a geometric form, as appeared in \eqref{eq:f-form} and \eqref{eq:f-simple} for the scalar function $f$. This in turn, lead us to the important geometric observation (about preserved geodesics) of item (\ref{path-char:item1}) of Theorem~\ref{thrm:geom-1}.}
\end{remark}
\section{The geometry of the underlying space} 
\par Using the coordinate-free formulation, we can further investigate the structure of local regular solutions.  
\begin{theorem}[Geometry of local regular solutions-II]\label{thrm:structure}
	Suppose $\sigma$ solves the PDE \eqref{eq:Fins-main-1}. Around a regular point $x$ of $\sigma$, the following properties hold true for the (regular) level set $\Sigma_x:= \sigma^{-1}(\sigma(x))$.
	\begin{enumerate}
	\item  \label{structure:item1}
		\makebox[\linewidth]{\(\begin{aligned}[t]
				\partial_1 g_{\alpha \beta} = 2 \left(\nicefrac{f}{\|\grad \sigma\|}\right)g_{\alpha\beta}.
			\end{aligned}		\)}
		\smallskip
	\item \label{structure:item2} The symbols $\overstar{\Gamma}^i_{jk}$ satisfy
		\begin{align*}
			\overstar{\Gamma}^1_{11} = \overstar{\Gamma}^1_{1\alpha} = \overstar{\Gamma}^\alpha_{11} =  0, \quad 
			\overstar{\Gamma}^\alpha_{1 \beta} = - \left(\nicefrac{\rho''}{\rho'}\right) \delta^\alpha_\beta, \quad \text{and} \quad
			\overstar{\Gamma}^1_{\alpha \beta} = \left(\nicefrac{\rho''}{\rho'}\right)g_{\alpha\beta}.
		\end{align*} 
		\smallskip
	\item \label{structure:item3}
		The flag curvatures $\K({\sf n}, X) = \mathrm{sec}_{g^{\sf n}}(n,X)$ are equal to $-\nicefrac{\rho'''}{\rho'}$, in particular, they are locally constant along $\Sigma$. 
		 \smallskip
		 \item  \label{structure:item4} There exists a Riemannian metric $g_{\sf level}$ on $\Sigma$ (independent of $u$) such that, $g^{\sf n}$ is of the warped product form
		 \begin{align}\label{eq:warped}
		 	g^{\sf n} = du^2 + (e^{-\sigma}\sigma'(u))^2 g_{\sf level} = du^2 + \left(\rho'\right)^2 g_{\sf level}.
		 \end{align}
	\end{enumerate}
\end{theorem}
\begin{proof}[\footnotesize \textbf{Proof}]
\phantom{}\hfill
\begin{enumerate}
	\item []{\bf{\footnotesize \textsf{Proof of (\ref{structure:item1})}}}.
\par Using $\partial_1 = \delta_1 + G^r_1\dt{\partial}_r$, along with the h-metricity of the Cartan connection, and the fact that $\sigma_1$ is constant along $\Sigma$ we obtain
\begin{align*}
	\partial_1 g_{ij} &= \nabla_{\delta_1 + G^r_1\dt{\partial}_r} \left< \partial_i, \partial_j   \right>_{\y}\\
	&= \left< \nabla_{\delta_i} \partial_1, \partial_j   \right>_{\y} + \left< \partial_1, \nabla_{\delta_j}\partial_1   \right>_{\y} + G^r_1 \Big( \left< \nabla_{\dt{\partial}_r}\partial_i, \partial_j   \right>_{\y} + \left< \partial_i, \nabla_{\dt{\partial}_r} \partial_j   \right>_{\y} \Big)\\
	&= \left< \nabla_{\delta_i} \left(\nicefrac{\grad \sigma}{\sigma_1}\right), \partial_j   \right>_{\y} + \left< \partial_1, \nabla_{\delta_j} \left(\nicefrac{\grad \sigma}{\sigma_1}\right)   \right>_{\y},
\end{align*}
in which, we have used item (\ref{path-char:item5-(c)}) from Theorem~\ref{thrm:geom-1} to obtain the last equality. 
\par Hence, using the PDE \eqref{eq:Fins-main-1}, we get
\begin{align*}
	\partial_1 g_{\alpha \beta} = 2 \left(\nicefrac{f}{\sigma_1}\right)g_{\alpha\beta}
	= 2 \left(\nicefrac{f}{\|\grad \sigma\|}\right)g_{\alpha\beta}.
\end{align*}
\smallskip	
\item []{\bf{\footnotesize \textsf{Proof of (\ref{structure:item2})}}}.
\par The symbols $\overstar{\Gamma}^i_{jk}$, by computing the formal Christoffel symbols and using \eqref{eq:Chris-symbs} (or by direct computation!), are as follows:
\begin{align*}
	\overstar{\Gamma}^1_{11} = {\Gamma}^1_{11}  - g^{11}C_{111}G^1_1 =  {\Gamma}^1_{11} = \nicefrac{1}{2}\,g^{11}\partial_1 g_{_{11}} = 0, 
\end{align*}
\begin{align*}
	\overstar{\Gamma}^\alpha_{11} = {\Gamma}^\alpha_{11} + g^{\alpha \gamma} C_{111}G^1_\gamma = {\Gamma}^\alpha_{11} =  \nicefrac{1}{2} \,g^{\alpha \beta} \partial_\beta g_{_{11}} = 0,
\end{align*}
\begin{align*}
	\overstar{\Gamma}^1_{1\alpha} = {\Gamma}^1_{1\alpha} = \nicefrac{1}{2} \, g^{11} \partial_\alpha g_{_{11}} = 0,
\end{align*}
\begin{align*}
	\overstar{\Gamma}^1_{\alpha \beta} = {\Gamma}^1_{\alpha\beta} + g^{11} C_{\alpha\beta r}G^r_1 =  -\nicefrac{1}{2} \, \partial_1 g_{\alpha \beta} + C_{\alpha\beta \gamma}G^\gamma_1 = \left(\nicefrac{\rho''}{\rho'}\right) \; g_{\alpha\beta},
\end{align*}
\begin{align*}
	\overstar{\Gamma}^\alpha_{1 \beta} = {\Gamma}^\alpha_{1\beta} - g^{\alpha\gamma} C_{\beta\gamma\eta} G^\eta_1 = \nicefrac{1}{2} \, g^{\alpha \eta} \partial_1 g_{\beta \eta} - g^{\alpha\gamma} C_{\beta\gamma\eta} G^\eta_1 = - \left(\nicefrac{\rho''}{\rho'}\right) \; \delta^\alpha_{\beta},
\end{align*}
where, the last two lines follow from $G^\alpha_1 \equiv 0$ combined with item (\ref{structure:item1}) above. 
\smallskip 
\item []{\bf{\footnotesize \textsf{Proof of (\ref{structure:item3})}}}. 
\par By \eqref{eq:Fins-main-1}, we have
\begin{align*}
	\nabla_{\delta_i} (\grad \sigma) = \sigma_i\grad\sigma + f\partial_i.
\end{align*}
\par As a result,  
\begin{align*}
	\Riem(\partial_i,\partial_j)\grad \sigma &=
	\Omega(\delta_i,\delta_j)\grad \sigma\notag\\
	&= \nabla_{\delta_i}\nabla_{\delta_j} \grad \sigma - \nabla_{\delta_j}\nabla_{\delta_i} \grad \sigma - \nabla_{\left[\delta_i, \delta_j \right]} \grad \sigma \notag\\
	&= \nabla_{\delta_i}\left( f\partial_j + \sigma_j \grad\sigma  \right) - \nabla_{\delta_j}\left( f\partial_i + \sigma_i \grad\sigma  \right) - \nabla_{\left[\delta_i, \delta_j \right]} \grad \sigma \\
	&= f_i\partial_j - f_j\partial_i  + f\left( \overstar{\Gamma}^k_{ji} - \overstar{\Gamma}^k_{ij} \right) \partial_k  - \nabla_{h\left[\delta_i, \delta_j \right]} \grad \sigma \notag \\
	& \phantom{sajjad} + \sigma_j\left( \sigma_i\grad\sigma + f\partial_i \right) - \sigma_i\left( \sigma_j\grad\sigma + f\partial_j \right) \notag\\
	&= f_i\partial_j - f_j\partial_i + f\sigma_j\partial_i -f\sigma_i\partial_j\notag.
\end{align*}
This yields
\begin{align*}
\Riem(X, \partial_u) \partial_u = \sigma_u^{-1}\Big(- f_u + f\sigma_u\Big)X;
\end{align*}
therefore, for $\|X\|_{\sf n} = 1$, we get
\begin{align*}
\K({\sf n}\wedge X) = \left<\Riem^{\sf n}(X, \partial_u) \partial_u, X  \right>_{\sf n} = \sigma_u^{-1}\Big(- f_u + f\sigma_u\Big);
\end{align*}
see~\eqref{eq:flag}.  
\smallskip
\item []{\bf{\footnotesize \textsf{Proof of (\ref{structure:item4})}}}.
Now we show 
$
(e^{-\sigma}\sigma'(u))^{-2} g^{\sf n}_{\alpha \beta},
$
is independent of $u$.  
\begin{align}\label{eq:var-sep}
	\partial_u \Big(\nicefrac{g^{\sf n}_{ij}}{e^{-2\sigma}\left(\sigma'\right)^2}\Big) &= \Big(\nicefrac{\left(2e^{-2\sigma}f\sigma_u^{-1}(\sigma_u)^2 + 2e^{-2\sigma}(\sigma_u)^3 - 2e^{-2\sigma}\sigma_u\sigma_{uu}\right)}{\left(e^{-4\sigma}(\sigma_u)^4\right)}\Big) \; g^{\sf n}_{ij} \notag \\
	&= \left( \nicefrac{\left(2f + 2(\sigma_u)^2 - 2\sigma_{uu}\right)}{\left(e^{-2\sigma}(\sigma_u)^3\right)}\right) \; g^{\sf n}_{ij}\notag\\
	&= 0,\notag
\end{align}
where the first equality follows from item (\ref{structure:item1}) above and the last equality follows from item (\ref{path-char:item2}) of Theorem~\ref{thrm:geom-1}.   
\par We denote this constant metric by $g_{\sf level}$. Consequently, $g^{\sf n}$ is expressed as the warped product
\begin{align*}
	g^{\sf n} = du^2 + (e^{-\sigma}\sigma')^2g_{\sf level} = du^2 + \left(\rho'\right)^2g_{\sf level};
\end{align*}
in particular, ${\sf n}$ is a unit geodesic field for $g^{\sf n}$. 
\end{enumerate}	
\end{proof}
\par \dunderline{1.2pt}{\small \textit{\textbf{\textsf{Terminology:}}}} \textsl{For a unit vector field $V$, the resulting Riemannian metric $g^{V}$ (e.g. the metric $g^{\sf n}$ above) is called an osculating metric; see e.g. \cite[Chapter 5]{Oh}}. 

\par \dunderline{1.2pt}{\small \textit{\textbf{\textsf{Notation:}}}} \textsl{In what follows, $\mathcal{C}$ denotes the set of critical point of $\sigma$ (and $\rho$).} 
\begin{theorem}[Critical points]\label{thrm:crit-structure}
	Let $e^{\sigma}F$ be a nontrivial conformal \cpt. Then the critical points of $\sigma$ are isolated.
\end{theorem}
\vspace{-2mm} 
\begin{proof}[\footnotesize \textbf{Proof}]
	By the local constancy of $\sigma$ and $F(\grad \sigma)$ along regular level sets, as shown in Theorem~\ref{thrm:structure}, we know that if $\mathcal{C}$ is a connected component of $\sigma^{-1}(a)$ containing a regular (critical) point, then it must consist entirely of regular (critical) points. As a result, the regular level sets are separating sets; that is they are hypersurfaces with two sides.  
\par Now, contrary to the statement, suppose the set of critical points of $\sigma$, denoted by $\mathcal{C}$, has a limit point. Let $\mathcal{C}_{\sf limit} \neq \varnothing$ denote the limit set of the critical points. Therefore, $\mathcal{C}_{\sf limit}$ is a closed set. Since the manifold is connected, we deduce that $\partial  \mathcal{C}_{\sf limit} \neq  \varnothing$. Let $p\in\partial \mathcal{C}_{\sf limit}$ and $o$ be a regular point nearby within the (forward) injectivity radius. Then, there is a unique geodesic $\eta$ from $p$ to $o$.
		\par Let $q$ be the farthest critical point on $\eta$ from $p$ (w.r.t. forward distance) so that the critical points in the interval $[p,q]$ are dense in $[p,q]$. In particular, the interval $(q,o)$ if not empty, contains an open sub-interval of the form $(q,a)$ that is free of critical points and that $a$ is a regular point. Such a point $q$ exists because the set of regular points is open. Moreover, this $q$ is a limit point for the critical set, i.e., $q \in \partial \mathcal{C}_{\sf limit}$. Henceforth, we adjust the original point $p$ to the position $q$ and adjust $o$ to the position $a$. Note that all distances are measured from $p$; that is, we use the forward distance form $p$.   
		\par By the previous argument, we have chosen a boundary limit critical point $p$ that is joined to a nearby regular point $o$ via a geodesic that contains no interior critical points. And $o$ is within the forward injectivity radius at $p$.
		\par First, assume $p$ is not a maximum point for $\sigma$. The previous argument applies when, shooting in any direction from $p$; hence, if $p$ is not a maximum point for $\sigma$, we can assume (by adjusting $o$), near $o$, $\sigma$ is increasing along the unique geodesic from $p$ to $o$. Let us denote this geodesic by $\mu(s)$; also, we set $\mu(0) = p$.
		\par By the argument at the beginning of this section, for $s>0$, the connected component of $\sigma^{-1}(\sigma(\mu(s)))$ containing the point $\mu(s)$, is entirely regular  and is a closed hypersurface that is also separating; we denote this hypersurface by $\Sigma(s)$. 
		\par Let $z$ be the closest point on $\Sigma_o := \Sigma_{d(p,o)}$ from $p$, then $z$ lies within the injectivity radius at $p$. Let $\eta$ be the unique geodesic from $p$ to $z$. $\eta$ must intersect all the hypersurfaces $\Sigma_s$ since these surfaces are separating sets (that separate $p$ from $\Sigma_o$). 
		\par One can parameterize $\eta$ via the inclusion $\eta(s) \in \Sigma_s$. By the Gauss lemma (e.g. see \cite[Lemma~3.18]{Oh}), $\eta$ coincides with one of the forward \emph{normal geodesics} (see item (\ref{path-char:item1}) of Theorem~\ref{thrm:geom-1}) to the connected level sets $\Sigma_s$ (along which $\sigma$ is increasing) thus the parametrization $\eta(s)$ is indeed the arc-length.
		\par {\footnotesize \sf \textbf{Claim.}} $\eta(0) = p$. 
		\par {\footnotesize \sf \textbf{Proof of the claim:}} Clearly, by the warped product decomposition \eqref{eq:warped}, one has
		\begin{align*}
			\dist_{g^{\sf n}}(\eta(s), \mu(s)) \le \dist_{(\Sigma(s), g_{\sf level})} (\eta(s), \mu(s)) \le \frac{e^{-\sigma(s)}\sigma_u(s)}{-\sigma(d(p,o)) \sigma_u(d(p,o))}d(z,o);
		\end{align*}
		hence,
		$
		\lim_{s \downarrow 0}\dist_{g^{\sf n}}(\eta(s), \mu(s)) = 0.
		$
		which means that $\eta(0)$ must coincide with $\mu(0) = p$. \scalebox{0.6}{$\blacksquare$}
		\par This also shows the level sets $\Sigma_s$ for $s>0$ are geodesic spheres (w.r.t. the forward distance) centered at $p$ both w.r.t. $F$ and w.r.t. $g^{\sf n}$. In other words, $s$ is just the geodesic distance from $p$. Hence, $p$ can be the only critical point in a normal neighborhood around $p$ i.e. $p$ is an isolated point of $\mathcal{C}$. This contradicts the assumption $p \in \partial \mathcal{C}_{\sf limit}$. 
	\par It remains to argue a contradiction if $p$ is indeed a maximum point of $\sigma$. In this case the argument is essentially the same. We assume $\sigma(p) = 0$, use the backward distance from $p$, utilize the Gauss lemma to get a geodesic $\eta(s)$ ($s<0$) and using the fact that normal geodesics are reversible (see item (\ref{path-char:item2}) of Theorem~\ref{thrm:geom-1}), we reach the conclusion that $p$ is an isolated critical point which is a contradiction. 
\end{proof}
\section{Berwaldian rigidity}
\par In this section, $(M,F)$ is taken to be Berwaldian and admitting a nontrivial Finslerian \cpt,  $\bar{F} = e^{\sigma} F$; \emph{it is important to note that we are \underline{not} assuming that $\bar{F}$ is Berwaldian.}  
\par Now, we can present a proof of our main Theorem: ``\emph{a complete Berwaldian manifold $\left(M,F\right)$ admitting a nontrivial conformal circle preserving transformation (\cpt for short) must be Riemannian, provided that it has a dense subset on which no flag curvature vanishes.}'' 
\begin{remark}
	As we will see below, the set of points where some flag curvatures vanish ($\mathcal{Z}$), is a closed subset; thus, being nowhere dense will be equivalent to its complement being a dense open subset. 
\end{remark}
\subsection*{Proof of the Main Theorem}
By items (\ref{structure:item1}) and (\ref{structure:item2}) of Theorem~\ref{thrm:structure} and by the Berwaldian hypothesis, we know that on the regular set,
\begin{align*}
\overstar{\Gamma}^1_{\alpha \beta} = - \nicefrac{1}{2}\, \partial_1 g_{{\alpha \beta}} = - \left(\nicefrac{f}{F(\grad \sigma)} \right) g_{\alpha\beta} = \left(\nicefrac{\rho''}{\rho'}\right) g_{\alpha\beta},
\end{align*}
is independent of $\y$. Hence, when $\rho',\rho'' \neq 0$, we deduce that $g_{\alpha\beta}$ is independent of $\y$ and as are result, so is $g$, i.e., $g$ is Riemannian. 
\par It remains to show $\y$-independence globally. Set
\begin{align*}
\mathcal{S} := \left(\rho''\right)^{-1}(0) \cup \left(\rho'\right)^{-1}(0).
\end{align*}
 By smoothness of the Finslerian metric in $x$, to reach the conclusion, it is sufficient to argue that $\mathcal{S}^{c}$ is dense. 
 \par Since by Theorem~\ref{thrm:crit-structure}, $\mathcal{C} = \left(\rho'\right)^{-1}(0)$ is a discrete set, to show $\mathcal{S}^{c}$ is dense, we only need to show $\left(\rho''\right)^{-1}(0)$ is a nowhere dense set. 
\par Set
\begin{align*}
\mathcal{Z} := \left\{ p \in M \; \text{\textbrokenbar} \; \text{some flag curvatures vanish at $p$}  \right\}.
\end{align*}
By the hypothesis, the set, $\mathcal{Z}$, is a nowhere dense set. Also considering the flag curvature as fiber-wise real function on the bundle $UM \wedge UM$, the set of $(x,X,Y)$ for which $\K(X\wedge Y) = 0$ is closed therefore, its projection to $M$, which is $\mathcal{Z}$, is also closed (as projections preserve closedness). Thus, we deduce that $\mathcal{Z}$ is a closed nowhere dense set. 
\par Consider the dense open set 
\begin{align*}
\mathcal{O} := \left(M \smallsetminus \mathcal{Z}\right) \smallsetminus \mathcal{C}.
\end{align*}
One has
\begin{align*}
-\nicefrac{\rho'''}{\rho'} = \K({\sf n} \wedge X) \neq 0\quad \text{on}\quad \mathcal{O}
\end{align*}
i.e. $\rho''' = \grad \left(\rho''\right) \neq 0$ on $\mathcal{O}$; meaning $\rho''$ is regular on $\mathcal{O}$ (of constant rank $1$). 

\par By the regular level set theorem (or the constant rank theorem), $\left(\rho''\right)^{-1}(0) \cap \mathcal{O}$ is either empty or is an embedded submanifold of co-dimension $1$; in any case, it is a null-measure subset in $\mathcal{O}$ (in particular, is nowhere dense in $\mathcal{O}$) and consequently so is in $M$. Consequently, $\mathcal{S}$ is nowhere dense as it is the union of two such sets. This finishes the proof.
\newline\phantom{}\hfill\scalebox{1.2}{\qed}
\section{Further remarks}
\begin{enumerate}
		\item [1--] 
		\par There exist erroneous treatments of the Finslerian PDE \eqref{eq:Fins-cpt-horiz-hess} in some -- not further specified -- literature that simply overlook the dependence of the spray coefficients on the direction, by claiming that $G^i \equiv 0$. To no surprise, this would lead to very Riemannian-like results. 
		\par Explicitly speaking, we note that the spray coefficients $G^i$ do not necessarily vanish in directions other than ${\sf n}$ (see Theorem~\ref{thrm:structure}), indeed $G^i \equiv 0$ is the characteristic of flat Finslerian metrics, for instance, any Berwaldian such manifold must have vanishing curvature \cite{Raps}. Now, considering a sphere (not flat obviously) in the Euclidean space, the inversion w.r.t. any point outside the sphere is a nontrivial conformal map of spheres, thus, is circle-preserving as well; e.g. see~\cite{Kuh}.  
		\smallskip
	\item [2--] 
	\par The recent article \cite{TN} claims to have established that every Berwaldian manifold with nowhere vanishing flag curvature (so either positive or negative flag curvature) is already Riemannian , hence, extending the well-known rigidity result of Szab\'o on Berwaldian surfaces~\cite{Szabo}. 
	\par Our result, in contrast, allows for a very large set on which some flag curvatures vanish (the bad set), meaning, we allow sign change in the flag curvature and the price to pay to get rigidity, is requiring a nontrivial \cpt. The bad set for example can be a fat fractal or Cantor-like subset, as long as the complement remains dense.
	\smallskip
	\item [3--] 
	\par Since the horizontal covariant differentiation w.r.t. Cartan and Chern connections coincide, the coordinate-free characterization \eqref{eq:Fins-main-1}, holds verbatim in terms of the Chern connection.
	\smallskip  
	\item [4--] 
	\par Clearly a weaker version of our main theorem is: ``\emph{A complete Berwaldian manifold with almost everywhere non-vanishing flag curvatures, cannot admit a nontrivial $\cpt$ unless it is Riemannian.}'' 
\end{enumerate}

\end{document}